\documentclass[12pt]{article}
\usepackage{latexsym}
\usepackage{amssymb,amsfonts,enumerate}
\usepackage{amsmath}
\newtheorem{theorem}{Theorem}

\newtheorem{lemma}{Lemma}

\newtheorem{prop}{Proposition}
\newtheorem{cor}{Corollary}

\newenvironment{proof}{\medskip \noindent
{\bf Proof.}}{\hfill \rule{.5em}{1em}
\\}

\def\bea{\begin{eqnarray*}}
\def\eea{\end{eqnarray*}}
\def\be{\begin{equation}}
\def\ee{\end{equation}}

\begin{document}

\title{Vacuum static spaces  with harmonic curvature}

\author{ Jongsu Kim
\thanks{This research was supported by Basic Science Research Program through the National Research Foundation of Korea (NRF) funded by the Ministry of Science, ICT and Future Planning (NRF-2020R1A2B5B01001862).
} }



\maketitle

\begin{abstract}
 In this article we make a thorough classification of (not necessarily complete) $n$-dimensional vacuum  static spaces  $(M,g,f)$ with harmonic curvature and, as a corollary,  obtain  a classification of  complete  vacuum  static spaces  with harmonic curvature.
Indeed, we showed that $(M,g,f)$ is locally isometric to  one of the following four types;  (i)  the Riemannian product of an Einstein manifold and a vacuum static space,
(ii)  the warped product over an interval with an Einstein manifold as fiber,
(iii)  an  Einstein manifold,
(iv) the Riemannian product of two Einstein manifolds.
\end{abstract}



\newtheorem{conj}[theorem]{Conjecture}

\section{Introduction}

An  $n$-dimensional Riemannian manifold $(M,g)$ is said to be a vacuum static space if it admits
a non-constant smooth function $f$  satisfying the following equation
\begin{eqnarray} \label{0002bx}
\nabla df = f(r -\frac{R}{n-1} g)  ,
\end{eqnarray}
where $\nabla $, $r$ and  $R$ is the Levi-Civita connection,  Ricci tensor and  scalar curvature of $g$, respectively.

Vacuum static spaces are closely related to  static space-times satisfying the vacuum Einstein field equation in general relativity \cite{HE, KOb}.
Indeed, for a Riemannian manifold $(M, g)$ with a positive function $f$,  the Lorentzian metric $\hat{g}  = - f^2 dt^2 +  g$ on  $ I \times M  $ for an interval $I$
is Einstein
    if and only if (\ref{0002bx}) holds for $(M, g, f) \ $ \cite[Proposition 2.7]{Co}.
There is an extensive literature on vacuum static spaces.
Among them, a short list concerned with general relativity aspect is as follows:
\cite{ACD, BM, CS,  HE, Is, JEK, Li, Wad}.

\medskip
A fundamental Riemannian geometric importance of  vacuum static spaces comes from the fact that
the derivative of the scalar curvature functional, deﬁned on the space $\mathcal{M}$ of smooth Riemannian
metrics on a closed manifold $M$, is surjective at $g \in \mathcal{M}$  if there is no nonzero
smooth function $f$ on $M$ satisfying (\ref{0002bx}); see \cite[Chapter 4]{Be} and  \cite{Bou, FM}.
A local version of this fact still holds in appropriate sense;  see \cite[Theorem 1]{Co} when $M$ is
a compactly contained subdomain  of a smooth manifold.
This local viewpoint has been
studied to make remarkable progress in Riemannian and Lorentzian geometry \cite{CIP,   CEM,  QY2}.

\smallskip
It is interesting to study the classification of local and complete vacuum static spaces under natural geometric or topological assumptions.
 Kobayashi \cite{Ko} and Lafontaine \cite{La}
independently proved a classiﬁcation of closed locally conformally ﬂat vacuum static spaces; all the local ones are also thoroughly described  in  \cite{Ko}.
Qing and Yuan \cite{QY2} classiﬁed complete Bach-ﬂat vacuum static spaces  with compact
level sets of $f$; in particular these spaces have  harmonic curvature and have at most two distinct Ricci-eigen values at each point. Baltazar, Barros, Batista and Viana \cite{BBBV}  and Ye \cite{Ye} proved that $n$-dimensional, $n\geq 5$, compact vacuum static spaces with positive scalar curvature and zero radial Weyl curvature are Bach-ﬂat. Hwang and Yun \cite{HY2} showed that $n$-dimensional, $n\geq 5$,
vacuum static spaces with vanishing of complete divergence of Bach tensor
and Weyl tensor have harmonic curvature if it has compact
level sets of $f$.


Meanwhile, for closed three dimensional vacuum static spaces, Ambrozio \cite{Am} studied topological and geometrical properties; see also Shen \cite{She}.


\medskip

In this paper we intend to classify vacuum static spaces with harmonic curvature.
The harmonic curvature condition is  interesting for its own sake. Riemannian manifolds with harmonic curvature are characterized when they admit at most two distinct Ricci eigenvalues at each point \cite[Chap. 16]{Be} and  \cite{De}.

Four dimensional vacuum static spaces  with harmonic curvature  have been classified  in \cite{KS}.
The approach therein
depends on eigenvalue analysis based on the Codazzi tensor $r$. Although it was effective enough to yield explicit local and global description of the spaces, the computation becomes much harder in higher dimension, leaving an interesting challenge. Here we overcome the difficulty and achieve a thorough classification without dimensional restriction, Recall that  a $3$-dimensional Riemannian manifold $(M, g)$ has  harmonic curvature if and only if it is locally conformally flat.

\begin{theorem} \label{locals}
Let $(M^n, g, f)$ be a (not necessarily complete) $n$-dimensional, $n \geq 3$, vacuum static space with harmonic curvature.
Then for each point in some open dense subset of $M$, there exists a neighborhood $V$ such that
 one of the following four assertions holds:

{\rm  (i)}  there are Einstein manifolds $(N^{k-1},\tilde{g}_1)$ and $(U^{n-k},  \tilde{g}_2)$ with the Ricci tensor relation $r^{\tilde{g}_1} =  (k-2) k_2 \tilde{g}_1$ and $r^{\tilde{g}_2} =  (n-k-1) k_n \tilde{g}_2$ for constants $k_2, k_n$ such that

$(V, g)$ is isometric to a domain in the Riemannian product of $(N^{k-1}, p^2 \tilde{g}_1)$ for a constant $p$ and a vacuum static space $(W^{n-k+1}:=I \times U^{n-k}, \ ds^2+h(s)^2\tilde{g}_2, \ f) $ which is a warped product over an open interval $I$; the function  $h$ on $I$ satisfies with constants $c_2$ and $c_3$
\begin{eqnarray} \label{e2}
 \ \ \ (h')^2  + \frac{Rh^2}{(n-1)(n-k+1)}   +\frac{2c_2h^{-n+k+1}}{(n-k-1)}=k_n, \  {\rm for}\  \ n-k \geq 2,  \\
(h')^2  + \frac{R}{2(n-1)} h^2  =c_3,  \ \ \ \ {\rm for} \ n-k =1. \  \nonumber  \hspace{3cm}
\end{eqnarray}
 Moreover,  $f = c \cdot h^{'}(s)$ for a constant $c$,  and  $p$ satisfies
  \begin{eqnarray}
\frac{(k-2)}{p^2} k_2  = \frac{R}{n-1}.
\end{eqnarray}

\medskip
{\rm  (ii)} $(V, g)$ is isometric to a domain in the warped product $( I \times N^{n-1},  ds^2+h(s)^2g_N) $ where $g_N$ is an Einstein metric with the Ricci tensor $r^{g_N} = (n-2)k g_N$ for a constant $k$. The function $h$  satisfies $(h')^2  + \frac{R}{n(n-1)} h^2  +2\frac{a}{(n-2)h^{n-2}}=k$, for a constant $a$.

\medskip
{\rm  (iii)}
$(V, g)$ is Einstein and a warped product over an interval;  $g=ds^2 + (f^{'}(s))^2 \tilde{g}$, where  $ \tilde{g}$ is Einstein.  And $f$ satisfies
$ f^{''} = - \frac{R }{n(n-1)}f$.

\medskip
{\rm  (iv)} $(V, g)$ is isometric to a domain in the Riemannian product of two Einstein manifolds  $(N_1^k , g_1) $  and  $(N_2^{n-k} , g_2) $ with their Ricci tensors satisfying $r_{g_1} = (1-  \frac{1}{k} ) \frac{R}{n-1} g_1$ and  $r_{g_2} =  \frac{R}{n-1} g_2$. Moreover, $f'' =   - \frac{R}{k(n-1)} f$.


\end{theorem}

We note that most of the spaces in the case (i) have three distinct Ricci eigenvalues at some point and that some spaces in the case (iv) is not Bach-flat.
Kobayashi's paper \cite{Ko} has described  the above case (ii) in concrete detail
in the context of $g$ being locally conformally flat.
As a corollary to Theorem \ref{locals}, we obtain the classification of complete vacuum static space with harmonic curvature.

\begin{cor} \label{complete}
Let $(M^n, g, f)$ be a  complete $n$-dimensional vacuum static space with harmonic curvature.
Then it is one of the following:

{\rm  (i)}
$(V, g)$ is isometric to a quotient of the Riemannian product of an Einstein manifold  $(N^{k-1}, {g}_1)$  and a vacuum static space $(W^{n-k+1}:=\mathbb{R}^1 \times U^{n-k}, \ ds^2+h(s)^2\tilde{g}_2, \ f) $ which is a warped product over $\mathbb{R}^1$.

\medskip
{\rm  (ii)} $(V, g)$ is isometric to  a quotient of  the warped product $( \mathbb{R}^1 \times N^{n-1},  ds^2+h(s)^2g_N) $ over an interval $\mathbb{R}^1$ with Einstein manifold $ (N^{n-1} , g_N) $ as fiber.

\medskip
{\rm  (iii)}
$(V, g)$ is Einstein.

\medskip
{\rm  (iv)} $(V, g)$ is isometric to a quotient of the Riemannian product of two Einstein manifolds  $(N_1^k , g_1) $  and  $(N_2^{n-k} , g_2) $.

\end{cor}

To prove Theorem \ref{locals}, we recall some argument of the four dimensional work \cite{KS}.
The gradient $\nabla f$ of the function $f$ is known to be an eigenvector field for the  Ricci tensor. We let $E_1 = \frac{\nabla f }{|\nabla f |}$ and form a Ricci-eigen orthonormal frame field $E_i$, $i=1, \cdots, n$ with
corresponding eigenvalues $\lambda_i$. We use the vacuum static space equation and the harmonic curvature equation to
write the Riemannian metric in a simple shape, following \cite{CC}.
 These two equations also provide a number of relations on  $\lambda_i$  and components $\Gamma_{ij}^k :=g(\nabla_{E_i} E_j,  E_k)$ of  the Riemannian connection.


Next, we use the refined adapted Ricci-eigen frame field  which has been developed by the author in  \cite{Ki4}. This refined field has simpler local expression and $\Gamma_{ij}^k$ is computed better. Then
the vacuum static space equation becomes tractable. The details of subsequent analyzing argument turn out to be fairly different from those of \cite{KS, Ki4}, and resolve the main difficulty arising from higher dimension.
 We could show that the number of distinct Ricci-eigenvalues at each point is at most three.
Theorem \ref{locals} is then proved by using the result of \cite{Ki2, Ki3} where vacuum static spaces with at most three Ricci eigenvalues are classified.
Corollary \ref{complete} on the  classification of complete ones follows by standard argument.

\bigskip
This paper is organized as follows. In Section 2,  we recall
some properties on vacuum static spaces with harmonic curvature. In Section 3, we derive various formulas for vacuum static space with harmonic curvature. In Section 4 we
recall refined adapted Ricci-eigen vector fields from \cite{Ki4}.
In Section 5 we
compute  $\Gamma_{ij}^k$ in refined adapted Ricci-eigen vector fields.
In Sections 6
 we compute out relations on $\lambda_i$ and $\Gamma_{ij}^k$  out of the vacuum static space equation and the harmonic curvature equation.
 In Section 7
we analyze in detail the relations of the previous section to prove the classification of local and complete vacuum static spaces with harmonic curvature.

\section{Preliminaries}

In this section we recall some known results about vacuum static spaces with harmonic curvature and add a little explanation as necessary. A Riemannian metric has harmonic curvature if and only if
the Ricci tensor  is a Codazzi tensor, written in local coordinates as $\nabla_k r_{ij}  = \nabla_i r_{kj}$, \cite[Chap. 16]{Be}.
A Riemannian manifold with harmonic curvature is real analytic in harmonic coordinates \cite{DG}.
Below we shall denote the Ricci tensor as $r$ or $R( \cdot , \cdot )$.

 \begin{lemma} \label{threesolx}
For an $n$-dimensional manifold $(M^n, g, f)$ with harmonic curvature satisfying {\rm (\ref{0002bx})}, it holds that
\begin{eqnarray*} \label{solbax}
-R(X,Y,Z, \nabla f) = - R(X,Z) g(\nabla f, Y)  + R(Y,Z)g(\nabla f, X) \ \ \ \ \ \ \ \ \ \ \ \ \\
 - \frac{R}{n-1} \{ g(\nabla f, X)  g(Y,Z) - g(\nabla f, Y) g(X,Z) \}   .
\end{eqnarray*}
\end{lemma}

\begin{proof}
The proof is the same as that of Lemma 2.2 of \cite{KS} except the difference of dimensions.
\end{proof}

 \begin{lemma}[Lemma 2.3, \cite{KS}] \label{threesolbx}
Let  $(M^n, g, f)$ have harmonic curvature, satisfying {\rm (\ref{0002bx})}. Let $c$ be a regular value of $f$ and $\Sigma_c= \{ x | f(x) =c  \}$  be the level surface of $f$. Then the following assertions hold:

{\rm (i)} Where $\nabla f \neq 0$,  $E_1 := \frac{\nabla f }{|\nabla f  | }$ is an eigenvector field of $r$.

{\rm (ii)} $ |\nabla f|$  is constant on a connected component of $\Sigma_c$.

{\rm (iii)} There is a function $s$ locally defined with   $s(x) = \int  \frac{   d f}{|d f|} $, so that

$ \ \ \ \ ds =\frac{   d f}{|d f|}$ and $E_1 = \nabla s$.

{\rm (iv)}  $R({E_1, E_1})$ is constant on a connected component of $\Sigma_c$.

{\rm (v)}  Near a point in $\Sigma_c$, the metric $g$ can be written as

$\ \ \ g= ds^2 +  \sum_{i,j > 1} g_{ij}(s, x_2, \cdots  x_n) dx_i \otimes dx_j$, where
    $x_2, \cdots  x_n$ is a local

 $ \ \ \ $   coordinates system on $\Sigma_c$.

{\rm (vi)}  $\nabla_{E_1} E_1=0$.
\end{lemma}

 For a point $x$ in $M$, let  $E_r(x)$ be the number of distinct eigenvalues of the Ricci tensor $r_x$,
and set $M_r = \{    x \in M \  |  \  E_r {\rm \ is \ constant \ in \ a \ neighborhood of \ } x \}$, following Derdzi\'{n}ski  \cite{De}, so that $M_r$ is an open dense subset of $M$.
 Then we have;

\begin{lemma}[Lemma 2.4, \cite{KS}] \label{abc60x} For a Riemannian metric $g$ of dimension $n \geq 4$ with harmonic curvature, consider orthonormal vector fields $E_i$, $i=1, \cdots n$ such that
$R(E_i, \cdot ) = \lambda_i g(E_i, \cdot)$. Then the followings hold
in each connected component of $M_r$;

\bigskip
\noindent {\rm (i)}
 $(\lambda_j - \lambda_k ) \langle \nabla_{E_i} E_j, E_k \rangle   + {E_i} \{R(E_j, E_k)   \}=(\lambda_i - \lambda_k ) \langle \nabla_{E_j} E_i, E_k\rangle + {E_j} \{R(E_k, E_i)   \}, \ \ $

 for any $i,j,k =1, \cdots n$.

\smallskip
\noindent {\rm (ii)}  If $k \neq i$ and $k \neq j$,
$ \ \ (\lambda_j - \lambda_k ) \langle \nabla_{E_i} E_j, E_k\rangle=(\lambda_i - \lambda_k ) \langle \nabla_{E_j} E_i, E_k\rangle .$

\smallskip
\noindent {\rm (iii)} Given distinct eigenfunctions $\lambda$ and $ \mu$ of the Ricci tensor $r$ and local vector fields $v$ and $ u$ such that  $r v = \lambda v$, $ru = \mu u$ with $|u|=1$, it holds

$ \ \ \ \ \  v(\mu) = (\mu - \lambda) <\nabla_u u, v > $.

\smallskip
\noindent {\rm (iv)} For each eigenfunction $\lambda$ of $r$, the $\lambda$-eigenspace distribution is integrable and its leaves are totally umbilic submanifolds of $M$.

\end{lemma}
We use the notation $\Gamma_{ij}^k = \langle \nabla_{E_i}E_j , E_k      \rangle$.
 Lemma \ref{threesolbx} implies that
for any point $p$ in the open dense subset $M_{r} \cap \{ \nabla f \neq 0  \}$ of $M^n$,
there is a neighborhood $U$ of $p$ where there exists an orthonormal Ricci-eigenvector fields $E_i$, $i=1, \cdots  , n$  such  that

{\rm (i)}  $E_1= \frac{\nabla f}{|\nabla f| }$,

 {\rm (ii)} for $i>1$, $E_i$ is tangent to smooth level hypersurfaces of $f$.

\medskip
These local orthonormal Ricci-eigenvector fields $\{ E_i \}$  shall be called an {\it adapted frame field} of $(M, g, f)$.
We set  $ \zeta_i:= - \Gamma_{ii}^1= - \langle   \nabla_{E_i}  E_i ,  E_1  \rangle=  \langle    E_i , \nabla_{E_i}  E_1  \rangle$, for $i >1$. By (\ref{0002bx}), $\nabla_{E_i}  E_1 = \nabla_{E_i} (\frac{\nabla f}{  | \nabla f |}) =   \frac{ f R({E_i}, \cdot) -  f\frac{R}{n-1} g( {E_i}, \cdot  )  }{  | \nabla f |} $. So we may write:
\begin{equation} \label{lambda06ax}
\nabla_{E_i}  E_1 =   \zeta_i E_i     \ \    {\rm where}   \   \zeta_i =     \frac{f R(E_{i}, E_i)  - \frac{R}{n-1}f  }{| \nabla f|}.
\end{equation}

Due to Lemma \ref{threesolbx}, in a neighborhood of a point $p \in M_{r} \cap \{ \nabla f \neq 0  \}$, $f$  may be considered as a function of the variable $s$ only, and $f^{'} := \frac{df}{ds} = |\nabla f|$.

\begin{lemma} \label{abc60byx} Let $(M, g, f)$ be an $n$-dimensional space  with harmonic curvature, satisfying  {\rm(\ref{0002bx})}.
The Ricci eigenfunctions $\lambda_i$ associated to an adapted frame field $E_i$
are constant on a connected component of a regular level hypersurface $\Sigma_c$ of $f$, and so depend on the local variable  $s$ only. Moreover, $\zeta_i$, $i=2, \cdots, n$, in  $\rm{(\ref{lambda06ax})}$ also depend on $s$ only.
 In particular, we have
$E_i (\lambda_j) = E_i (\zeta_k)= 0$ for $i,k >1$ and any $j$.
\end{lemma}
\begin{proof}
Lemma 3.1 of \cite{KS} gives the proof in dimensional four. Similar argument can be given in higher dimension. One can refer to Lemma 3 in \cite{Sh}.
\end{proof}

 \section{Vacuum static space with harmonic curvature}
For an adapted frame field $E_i$,  from (\ref{0002bx}), Lemma \ref{threesolbx} (vi)  and  (\ref{lambda06ax}) we have
 \begin{eqnarray}
f^{''}=  f ( \lambda_1   - \frac{R}{n-1}), \ \ \ \ \ \ \ \ \  \ \ \ \ \ \ \   \label{F32a2}  \\
\zeta_i f^{'}=  f ( \lambda_i   - \frac{R}{n-1}), \ \ {\rm for} \ i >1.  \ \ \ \label{F32b2}
\end{eqnarray}

By direct curvature computation we get
\begin{eqnarray} \label{x23}
R_{1ii1} := R(E_1, E_i, E_i, E_1)=-\zeta_i^{'}  -  \zeta_i^2.
\end{eqnarray}
Lemma \ref{threesolx} gives
\begin{eqnarray} \label{x2}
R_{1ii1} = - \lambda_{i} +  \frac{R}{n-1}, \ \ \ \   i >1.
\end{eqnarray}
From (\ref{F32b2})$\sim$(\ref{x2})
we have $\zeta_i f^{'}=  f ( \zeta_i^{'}  +  \zeta_i^2)$.
Setting $\zeta_i  = \frac{u_i^{'}}{u_i}$ with $u_i \neq 0$, we get    $  f u_i^{''}-f^{'}u_i^{'} =0$. Integration gives $u_i^{'} = C_i f$ for a constant $C_i$.   When  $\zeta_i \neq 0$, we may write  $\zeta_i  = \frac{f }{ \int_{s_0}^s f(u) du   + c_i}$ for constants $c_i$ and $s_0$. The function $s$ is defined in Lemma \ref{threesolbx} modulo a constant, so we set $s_0 =0$. Setting $h =   \int_{0}^s f(u) du$,
we have

\begin{eqnarray}
h^{'} =f, \label{x49} \hspace{3.4cm} \\
\zeta_i(s)  = \frac{h^{'}}{h+ c_i}  \ \ \ {\rm if} \   \zeta_i \neq 0. \label{x4} \hspace{1cm} \end{eqnarray}
From (\ref{F32b2}), (\ref{x23}) and (\ref{x2}),
\begin{eqnarray}
\lambda_i  =  \frac{R}{n-1} \ \ \ {\rm if} \   \zeta_i = 0,  \hspace{1.4cm} \label{x5} \\
\lambda_i  = \frac{h^{''}}{h+ c_i} +  \frac{R}{n-1} \ \ \ {\rm if} \   \zeta_i \neq 0  . \ \ \nonumber
\end{eqnarray}

\bigskip

Suppose that there exists some nonzero $\zeta_i$.
We can rearrange $E_i$ so that $\zeta_i \neq 0$ for $i=2, \cdots, m$ and $\zeta_i = 0$ for $i> m$. From
(\ref{F32a2}),  (\ref{x49}),  (\ref{x5}) and
$\lambda_1 =R-  \sum_{i=2}^m \lambda_i -  \sum_{i=m+1}^n \lambda_i $, we get $\lambda_1  =  -  \sum_{i=2}^m \frac{h^{''}}{h+ c_i}   =  \frac{h^{'''} }{ h^{'}} +  \frac{R}{n-1} .$ So,

\begin{eqnarray} \label{x6}
 \frac{h^{'''} }{ h^{'}} +  \frac{R}{n-1}+ \sum_{i=2}^m \frac{h^{''}}{h+ c_i}=0 .
\end{eqnarray}
Set $H:= \Pi_{i=2}^m (h + c_i)$.
Multiplying (\ref{x6}) by  $ h^{'} H$,  we get
 $   H \{ h^{'''} + \sum_{i=2}^m \frac{ h^{'} h^{''} }{h  +c_i }+  \frac{R}{n-1}h^{'} \} =0 $.
Then $   \{ H h^{''} + \frac{R}{n-1} Q(h) \}^{'}  =0 $ where $Q(h)$ is a polynomial function of $h$ such that
$\frac{dQ}{dh} = H$.
Integration gives
 $    H h^{''} + \frac{R}{n-1} Q(h)   =a $ for a constant $a$. We write

\begin{equation}     \label{ff01}
h^{''} + \frac{RQ}{(n-1)H}  =\frac{a}{H}.
\end{equation}

Multiply by $2h^{'}$ and integrate to get

\begin{equation}\label{ff02}
(h^{'}(s))^2- (h^{'}(0))^2  + \int_0^s 2 \frac{RQ}{(n-1)H} h^{'}(u) du = \int_0^s 2\frac{a}{H}h^{'}(u)du.
\end{equation}

  We define an equivalence relation $\sim$  on the set  $\{ 2, 3, \cdots, n \} $ as below;
 \begin{eqnarray} \label{x3j0dk}
i \sim j \ \  {\rm  iff} \ \  \zeta_i = \zeta_j.
\end{eqnarray}
and let $[i]$ denote the equivalence class of $i$.

 \section{Refined adapted Ricci-eigen vector fields}

 In \cite{Ki4}, {\it refined adapted Ricci-eigen vector fields} are defined.
For completeness sake we present them in this section.

Suppose that there is  an adapted frame field $E_1,  \cdots,  E_n$ on  an open set $V  \subset M_{r} \cap \{ \nabla f \neq 0  \}$. We consider the local one-parameter group action $\theta_t(p):= \theta(p, t )$  associated to $E_1= \frac{\nabla f}{|\nabla f| }$, i.e.  $\frac{d ( \theta(p, t))}{dt} =  E_1 (\theta(p, t))$.
We shall use the Lie derivative  $L_{E_1 } E_i (q) := (L_{E_1 } E_i) (q) =   \lim_{h \rightarrow 0} \frac{1}{h} \{ {E_i}(q)- {\theta_{h}}_* {E_i}({\theta_{-h} (q)}) \} $, where ${\theta_{t}}_*$ is the derivative of
 the map $\theta_t : p \mapsto \theta_t( p )$.

For a number $c$ such that $V \cap  f^{-1}(c) \neq \emptyset$,    set  $ W :=V \cap  f^{-1}(c) $.  We denote the   $\lambda_i$-eigenspace in $T_{q}M$ by $\Lambda_{i, q}$.

\begin{lemma} \label{L61}
For $i \geq 2$ and $p \in W$, suppose that  $\{ \theta_h(p) \ |  \ h \in [-2\epsilon, 2\epsilon] \} \subset V $ for a number $\epsilon>0$.
Then ${\theta_{t}}_* ({E_i}(p))$  belongs to  $\Lambda_{i,\theta_t(p)}$ for $ -\epsilon \leq t    \leq \epsilon$.
\end{lemma}

\begin{proof}
 For  $i \geq 2$, $j \geq 1$  define $\beta_{ij}(t)  =  g   ( {\theta_{t}}_* {E_i}({\theta_{-t} (q)})  , {E_j}(q)  )$. We compute
\begin{eqnarray*}\frac{d}{dt} \{ {\theta_{t}}_* {E_i}({\theta_{-t} (q)}) \} =  \lim_{h \rightarrow 0} \frac{1}{h} \{{\theta_{t}}_*  {\theta_{h}}_* {E_i}({\theta_{-t-h} (q)}) -   {\theta_{t}}_* {E_i}({\theta_{-t} (q)})\}   \hspace{3.2cm}  \\
=  {\theta_{t}}_* [ \lim_{h \rightarrow 0} \frac{1}{h} \{  {\theta_{h}}_* {E_i}({\theta_{-t-h} (q)}) -   {E_i}({\theta_{-t} (q)})\}]
=  - {\theta_{t}}_* [ L_{E_1 } E_i (\theta_{-t} (q))]. \end{eqnarray*}


By Lemma \ref{abc60x} (ii), Lemma \ref{threesolbx} (vi) and   (\ref{lambda06ax}),  $ \Gamma_{1i}^l=0$ if $l \notin [i]$. So,  we have $L_{E_1 } E_i = \nabla_{E_1 } E_i  - \nabla_{E_i } E_1 = \sum_{l \in [i]. l \neq i}   \Gamma_{1i}^l E_l -  \zeta_i E_i   $. Then,
\begin{eqnarray} \label{bet3}
\\
\beta_{ij}^{'} (t)  =  -g   (  {\theta_{t}}_* [ L_{E_1 } E_i (\theta_{-t} (q))]  , {E_j}(q)  )\hspace{5.8cm}\nonumber \\
 \ \ \ \ \ \ \  = - g   (  {\theta_{t}}_* [  \sum_{l \in [i]. l \neq i}   \Gamma_{1i}^l( \theta_{-t} (q)) {E_l }({\theta_{-t} (q)}) -  \zeta_i( \theta_{-t} (q)) {E_i}({\theta_{-t} (q)})]  , {E_j}(q)  ) \  \nonumber  \\
 \ \ \ \  \ \  \ \ \ \ = -\sum_{l \in [i]. l \neq i}  \Gamma_{1i}^l(\theta_{-t} (q)) \beta_{lj} (t)+ \zeta_i (\theta_{-t} (q))  \beta_{ij} (t). \hspace{4.1cm} \nonumber
\end{eqnarray}

We fix $ j \notin [i]$.   We view (\ref{bet3})
as  a system of  first order ODE (ordinary differential equations)  for $r$ functions $\beta_{i_1 j} (t)$, $\beta_{i_2j} (t)$,  $ \cdots,   \beta_{i_r j} (t)$ where  $E_{i_1}, \cdots ,  E_{i_r}$ is the basis of the $\lambda_i$-eigenspace.
As $\beta_{i_1j}(0)= \cdots = \beta_{i_r j}(0)  = 0 $,  by the uniqueness of the solution of the ODE,   $\beta_{ij} (t) =  g  ( {\theta_{t}}_* {E_i}({\theta_{-t} (q)}) , {E_j}(q)  )=0$.
If $q = \theta_{t} (p) $, then
 ${\theta_{t}}_*( {E_i}(p) )  $ is orthogonal to  ${E_j}(q) $.
 These $E_j(q)$,  $j \notin [i]$, span the subspace orthogonal to $\Lambda_{i, q}$.
So, ${\theta_{t}}_*( {E_i}(p)) \in \Lambda_{i, \theta_{t}(p)}$ for $p \in W$.
\end{proof}

 The function $s$ in Lemma \ref{threesolbx} is defined modulo a constant and we may set $s(p)  = 0 $ for $p \in W$.
The function $\zeta_i$ depends only on $s$ by Lemma \ref{abc60byx}.    As  $\frac{d}{dt}  \theta_t (p) = E_1 ( \theta_t (p) )= \nabla s ( \theta_t (p) ) $, we have   $s( \theta_t  (p)) = t$.
We write
$ \zeta_i(\theta_{t} (p))=\zeta_i(s(\theta_{t} (p))) =\zeta_i( t) $.

\begin{lemma} \label{L42}
For $i \geq 2$ and $p \in W$, suppose  $\{ \theta_h(p) \ |  \ h \in [-2\epsilon, 2\epsilon] \} \subset V $ for a number $\epsilon>0$.  Let $F_i$ be the vector field defined  by $F_i (\theta_{t} (p)) =e^{- \int_{0}^t   \zeta_i(v) dv }  {\theta_{t}}_* ({E_i}(p))$.
If  $E_{i_1}(p), \cdots ,  E_{i_r}(p)$ is the basis of $\Lambda_{i, p}$,
 then

$F_{i_1}(\theta_{t} (p)), \cdots ,  F_{i_r}(\theta_{t} (p))$
 form an orthonormal basis  for $\Lambda_{i, \theta_{t}(p)}$.

\end{lemma}
\begin{proof} For $i, j \in \{i_1, \cdots, i_r \}$,
we define  $\gamma_{ij}(t) =  g  ( {\theta_{t}}_* {E_i}(\theta_{-t} (q))  , {\theta_{t}}_* {E_j}(\theta_{-t} (q)) )$.
Then $\gamma_{ij}(0)  =\delta_{ij}$. As $\zeta_i = \zeta_j$ and $[i]=[j]$, computing as in the proof of Lemma \ref{L61} yields
\begin{eqnarray} \label{f78}
\label{r5}   \\
\gamma_{ij}^{'}(t)
=  2 \zeta_i(\theta_{-t} (q))\gamma_{ij}(t)- \sum_{l \in [i], l \neq i}   \Gamma_{1i}^l(\theta_{-t} (q)) \gamma_{lj}(t)
-\sum_{l \in [j], l \neq j}   \Gamma_{1j}^l(\theta_{-t} (q)) \gamma_{il}(t). \nonumber
\end{eqnarray}
 We view (\ref{f78}) as a system of  first order ODE for  the $r^2 $ functions $\gamma_{ij} (t)$, $i, j \in \{ i_1, \cdots i_r \}$ in the variable $t$.
 One can check that $\tilde{\gamma}_{ij}(t) := \delta_{ij}e^{ \int_{0}^t   2 \zeta_i(\theta_{-u} (q)) du }  $ is a solution  of (\ref{r5}).
As  $\gamma_{ij}(0)  = \tilde{\gamma}_{ij}(0)=\delta_{ij}$,  by  the uniqueness of the solution of the ODE, $\gamma_{ij}(t)=  g   ( {\theta_{t}}_* {E_i}(\theta_{-t} (q))  , {\theta_{t}}_* {E_j}(\theta_{-t} (q))  ) =   \delta_{ij} e^{ \int_{0}^t   2 \zeta_i(\theta_{-u} (q)) du }  $.
So,  $ e^{ \int_{0}^t  - \zeta_i(\theta_{-u} (q)) du } {\theta_{t}}_* {E_i}(\theta_{-t} (q))$, $i =i_1, \cdots, i_r, $ is an orthonormal basis   for  $\Lambda_{i, q}$. Set   $q= \theta_{t}(p)$ for $p  \in W$ and $v= t-u$.
Then
 $F_i (\theta_{t} (p)) :=e^{- \int_{0}^t   \zeta_i(\theta_v(p)) dv }  {\theta_{t}}_*( {E_i}(p))$ is  an orthonormal basis of $\Lambda_{i, \theta_{t}(p)}$.
\end{proof}

By Lemma \ref{L42}, we have a {\it refined} adapted frame field $F_1:=E_1, F_2, \cdots,   F_n$. For any point $p_0$ in $M_{r} \cap \{ \nabla f \neq 0  \}$, there exists a neighborhood of $p_0$ equipped with a  refined adapted frame field.
We compute, for $i >1$,
\begin{eqnarray*}
L_{F_1 }   F_i (\theta_t(p)) =   \lim_{h \rightarrow 0} \frac{1}{h} \{  {\theta_{-h}}_* {F_i}( \theta_{h+t} (p)) -   {F_i}(\theta_t(p))  \} \hspace{3.3cm} \\
=   \lim_{h \rightarrow 0} \frac{1}{h} \{  {\theta_{-h}}_*  (e^{ \int_{0}^{t+h}  - \zeta_i(\theta_v(p)) dv }  {\theta_{t+h}}_* {E_i}_{p})
-   e^{ \int_{0}^t  - \zeta_i(\theta_v(p)) dv }  {\theta_t}_* {E_i}_{p}\} \\
=   \lim_{h \rightarrow 0} \frac{1}{h} \{  e^{ \int_{0}^{t+h}-   \zeta_i(\theta_v(p)) dv }  {\theta_{t}}_* {E_i}_{p}
-   e^{ \int_{0}^t -  \zeta_i(\theta_v(p)) dv }  {\theta_t}_* {E_i}_{p}\}   \hspace{1.5cm} \\
=- \zeta_i(\theta_{t} (p)) F_i (\theta_{t} (p)). \hspace{7cm}
\end{eqnarray*}
Simply writing, we have
$L_{F_1 }   F_i=- \zeta_i F_i .$
As adapted Ricci-eigen vector fields, $F_i$ satisfy (\ref{lambda06ax}), so
$L_{F_1 }   F_i = \nabla_{F_1 }   F_i  - \nabla_{F_i }   F_1 =  \nabla_{F_1 }   F_i  -\zeta_i F_i $.
So, we get
\begin{equation} \label{e07u}
\nabla_{F_1 }   F_i=0 \ \ {\rm for}  \ \  i >1.
\end{equation}

\section{Computation of $\Gamma_{ij}^k$ in a refined adapted frame field}

For a point $p_0 \in M_{r} \cap \{ \nabla f \neq 0  \}$, we consider a small coordinate neighborhood $V$ of Lemma \ref{threesolbx} (v).  For $p \in V \cap f^{-1} ( f(p_0))  $ and $i>1$,  $ \frac{d}{ dt } \{ x_i (  \theta_t (p) ) \} = dx_i  (  \frac{d}{ dt }\theta_t (p)  ) =dx_i  (  E_1  )   =0  $ from Lemma \ref{threesolbx}.
So,  $x_i (  \theta_t (p) ) = x_i (   p )$,  as long as $ \theta_t (p) \in V$.
Then ${\theta_t}_* ( \frac{\partial}{\partial x_i}  )  =  \frac{\partial}{\partial x_i} $.
We can still set $s( p) = 0$ so that  $s( \theta_t  (p)) = t$.
We may write $E_i = \sum_{l=2}^n a_{il} (s, x)  \frac{\partial}{\partial x_l}$ where $s=x_1$ and $x= (x_2, \cdots, x_n)$. Then $E_i(p) = \sum_{l=2}^n a_{il} (0, x(p))  \frac{\partial}{\partial x_l}$.
We get
$  {\theta_t}_* {E_i}(p) =  \sum_{l=2}^n a_{il} (0, x(p))  \frac{\partial}{\partial x_l}$.
Set $e_i (s,x) := \sum_{l=2}^n a_{il} (0, x)  \frac{\partial}{\partial x_l}$.
We can write
\begin{equation} \label{e07}
F_i (s, x) =  e^{ \int_{0}^s  - \zeta_i(v) dv }   e_i (s, x),    \ \  {\rm for } \ i>1.
\end{equation}

We have
$[F_i,  F_j](s,x) = e^{ \int_0^s -(\zeta_i +\zeta_j ) (v) dv  } [e_i, e_j]      $.
Here $ [e_i, e_j]$ is tangent to  the level surfaces of $f$. We write
$[e_i, e_j] (s, x) = \sum_{l=2}^n \tilde{\gamma}_{ij}^l ( x)   e_l  \\ = \sum_{l=2}^n \tilde{\gamma}_{ij}^l ( x) e^{ \int_{0}^s   \zeta_l(v) dv }  F_l$ for a function $\tilde{\gamma}_{ij}^l ( x)$ of $x$. For $i, j, k \geq 2$,  we get
\begin{equation} \label{e01}
 g([F_i,  F_j] , F_k    )(s, x)
=e^{ \int_0^s -(\zeta_i +\zeta_j - \zeta_k )  (v) dv  }\tilde{\gamma}_{ij}^k ( x). \end{equation}
From Cheeger-Ebin p.2, we have
\begin{equation}\label{e02}
2g(\nabla_{F_i}F_j , F_k)= g([F_i,  F_j] , F_k    ) - g([F_i,  F_k] , F_j    ) - g([F_j,  F_k] , F_i    ). \end{equation}

If $[i]=[j]=[k]$,   from (\ref{e01}) and (\ref{e02}) we get $\Gamma_{ij}^k (s,x) =e^{\int_0^s -\zeta_i   (v)dv   } {\hat\gamma}_{ij}^k ( x) $, where  $\hat{\gamma}_{ij}^k (x)= \frac{1}{2} ( \tilde{\gamma}_{ij}^k ( x)  -  \tilde{\gamma}_{ik}^j ( x)  -  \tilde{\gamma}_{jk}^i  ( x)  ) $.

 From (\ref{x4}), $e^{\int_0^s -\zeta_i   (v)dv   } = e^{-  \int_0^s (  \frac{h^{'}}{ c_i + h })   (v)dv   } $. So we may write for a constant $a_i$;
\begin{eqnarray} \label{fe1}
e^{\int_0^s -\zeta_i   (v)dv   }  =\frac{ a_i  } {h+ c_i} \ \ {\rm if} \ \zeta_i \neq 0. \end{eqnarray}

When $[i]=[j]=[k] $ and $\zeta_i \neq 0$,  we get
$\Gamma_{ij}^k =\frac{{\gamma}_{ij}^k (x)  }{h+ c_i}  $, where we set ${\gamma}_{ij}^k:=a_i \hat{\gamma}_{ij}^k $.

When $[i]=[j]=[k]$   and $\zeta_i = 0$, we have $\Gamma_{ij}^k =\hat{{\gamma}}_{ij}^k   $. We may set ${\gamma}_{ij}^k  =\hat{{\gamma}}_{ij}^k   $. We have got


\begin{lemma} \label{L10}
If $[i]=[j]=[k]$,  $\Gamma_{ij}^k (s, x) = \frac{ {\gamma}_{ij}^k(x)}{ h(s)+ c_i  }$ when $\zeta_i \neq 0$ and $\Gamma_{ij}^k =   {\gamma}_{ij}^k(x)$ when $\zeta_i = 0$.
\end{lemma}

\bigskip
Let $i, j$ be an integer $2 \leq i, j \leq n$ with $[i] \neq [j]$.
We use $\nabla_{F_1} F_i =0$ for $i >1$ to compute
the following Jacobi identity;
\begin{eqnarray*}
0= [[F_i, F_j], F_1] + [[F_j, F_1], F_i] +  [[F_1, F_i], F_j] \hspace{3.1cm}  \\
= [\sum_{k >1}^n (\Gamma_{ij}^k -\Gamma_{ji}^k ) F_k, F_1] + [\zeta_j F_j , F_i] -  [\zeta_i F_i , F_j] \hspace{2.8cm}\\
= \sum_{k >1}^n (\Gamma_{ij}^k -\Gamma_{ji}^k )(\zeta_k    -    \zeta_i - \zeta_j      ) F_k  - \sum_{k >1}^n F_1 (\Gamma_{ij}^k -\Gamma_{ji}^k ) F_k .\hspace{1.9cm}
\end{eqnarray*}
From this we obtain, for $2 \leq k \leq n$,
\begin{equation} \label{eq5} F_1 (\Gamma_{ij}^k -\Gamma_{ji}^k )  =  (\zeta_k   -\zeta_i - \zeta_j )( \Gamma_{ij}^k -\Gamma_{ji}^k)  .
\end{equation}

We shall prove

\begin{lemma}\label{L11}
Let $[i], [j], [k]$ be pairwise different. Then the following holds for some function ${\gamma}_{ij}^k(x)$ of $x$;

{\rm  (i)} $\Gamma_{ij}^k = \frac{  {\gamma}_{ij}^k(x)}{ h+ c_i  }$,   if $\zeta_i \neq 0$,

{\rm  (ii)} $\Gamma_{ij}^k =    {\gamma}_{ij}^k(x)$, if  $\zeta_i =0$.


\end{lemma}

\begin{proof}
If $[i], [j], [k]$ are pairwise different,
Lemma \ref{abc60x} (ii) gives $ \Gamma_{ji}^k = \frac{\zeta_j -  \zeta_k }{\zeta_i - \zeta_k     } \Gamma_{ij}^k$, and
$ \Gamma_{ij}^k -  \Gamma_{ji}^k  = \frac{\zeta_i -  \zeta_j }{\zeta_i - \zeta_k     } \Gamma_{ij}^k
 $.
So, $ \Gamma_{ij}^k -  \Gamma_{ji}^k =0$ if and only if $ \Gamma_{ij}^k=0$. If  $ \Gamma_{ij}^k=0$, (i) and (ii) hold.
Now assume $ \Gamma_{ij}^k \neq 0$.

To prove (i);  If  $\zeta_i , \zeta_j , \zeta_k$ are all different from  $ 0$, then  from (\ref{x4}),

 \noindent $\frac{\zeta_i -  \zeta_j }{\zeta_i - \zeta_k     } = \frac{c_i -  c_j }{c_i - c_k     }  \frac{c_k + h }{c_j +h  }  $.
So,
$ \Gamma_{ij}^k -  \Gamma_{ji}^k  = \frac{\zeta_i -  \zeta_j }{\zeta_i - \zeta_k     } \Gamma_{ij}^k
  =  \frac{c_i -  c_j }{c_i - c_k     }  \frac{c_k + h }{c_j +h  }  \Gamma_{ij}^k
 $.

$ F_1 (\ln|\Gamma_{ij}^k -\Gamma_{ji}^k| )=   F_1 (\ln| \frac{c_i -  c_j }{c_i - c_k     }  \frac{c_k + h }{c_j +h  }  \Gamma_{ij}^k| ) =   F_1 (\ln|  h   + c_k|-  \ln| h   + c_j| + \ln |\Gamma_{ij}^k| )\\
= \zeta_k -  \zeta_j + F_1 \ln |\Gamma_{ij}^k|
  $.
Meanwhile, from (\ref{eq5}),
$ F_1 (\ln|\Gamma_{ij}^k -\Gamma_{ji}^k| ) = \zeta_k   -\zeta_i - \zeta_j
  $. By comparison we obtain
$F_1 \ln |\Gamma_{ij}^k| = - \zeta_i =   -\frac{h^{'}}{ h  + c_i}  $. Integrating this, we get
$\Gamma_{ij}^k =  \frac{  {\gamma}_{ij}^k(x) }{ h+ c_i}. $

Next, suppose  $\zeta_j =0$. Then
 $\frac{\zeta_i -  \zeta_j }{\zeta_i - \zeta_k     } =  \frac{c_k + h }{c_k - c_i  }  $.
Then, $ F_1 (\ln|\Gamma_{ij}^k -\Gamma_{ji}^k| )=   F_1 (\ln|   \frac{c_k + h }{c_k - c_i }  \Gamma_{ij}^k| ) =   F_1 (\ln|  h   + c_k| + \ln |\Gamma_{ij}^k| )
= \zeta_k + F_1 \ln |\Gamma_{ij}^k|.
  $ Comparing with (\ref{eq5}),
$ F_1 \ln |\Gamma_{ij}^k| = -\zeta_i $.
We get $ \Gamma_{ij}^k
=\frac{\gamma_{ij}^k  }{ h+ c_i}   $. As $ \Gamma_{ij}^k = -\Gamma_{ik}^j  $, we get (i).

\medskip
To prove (ii);  if  $\zeta_i =0$, then
 $\frac{\zeta_i -  \zeta_j }{\zeta_i - \zeta_k     } =  \frac{c_k + h }{c_j +h  }  $. Following the computation in case (i), we  get $\Gamma_{ij}^k =   {\gamma}_{ij}^k(x) . $



\end{proof}

We highlight the following from Lemma  \ref{abc60x} (ii), (iii) and Lemma \ref{abc60byx};
\begin{equation} \label{eq37}
\Gamma_{ij}^k=0  \ \  {\rm for}   \ \  [i]= [j] \neq [k].
\end{equation}

\medskip
Now suppose $[i]\neq[j]= [k]$.
From (\ref{eq5}) and  (\ref{eq37}), we get   $F_1 (\Gamma_{ij}^k  )  =     -\zeta_i ( \Gamma_{ij}^k )$. From this we obtain

\begin{lemma} \label{L12}
Suppose $[i]\neq[j]= [k]$.  Then the following holds for some function ${\gamma}_{ij}^k(x)$ of $x$;


{\rm  (i)} $\Gamma_{ij}^k = \frac{  {\gamma}_{ij}^k(x)}{ h+ c_i  }$,   if $\zeta_i \neq 0$,

{\rm  (ii)} $\Gamma_{ij}^k =    {\gamma}_{ij}^k(x)$, if  $\zeta_i =0$.

\end{lemma}

One may view that  the formulas in Lemma \ref{L10}$\sim$Lemma \ref{L12} and (\ref{eq37})   define  ${\gamma}_{ij}^k(x)$, $i,j,k >1$ uniquely, given $\{ F_i\}$.

\section{Analysis of the vacuum static space equation}
Let $(M,g,f)$ be an $n$-dimensional Riemannian manifold with harmonic curvature satisfying (\ref{0002bx}) with non-constant $f$.
\begin{lemma}\label{77bx01}
For a refined adapted frame field $\{F_j\}$
in an open subset $V$ of $M_{r} \cap \{ \nabla f \neq 0  \}$, the following holds; for $ i \geq 2$,
\begin{eqnarray}
 \ \ \ \  R_{ii}
= -\zeta_i^{'}  -  \zeta_i^2  + \sum_{j>1, j \neq i}^n \{  -\zeta_i \zeta_j + F_i \Gamma_{jj}^i + F_j \Gamma_{ii}^j -(\Gamma_{ii}^j)^2  -(\Gamma_{jj}^i)^2 \} \label{y01}  \\
-\sum_{j>1, j \neq i}^n \{ \sum_{k \neq 1, i, j } \Gamma_{jj}^k \Gamma_{ii}^k
  + \Gamma_{ij}^k\Gamma_{jk}^i +( \Gamma_{ij}^k-\Gamma_{ji}^k) \Gamma_{kj}^i \}. \hspace{1.2cm}  \nonumber
\end{eqnarray}

\end{lemma}





\begin{proof}
For distinct $i, j \in \{2,3, \cdots , n \}$, we can directly compute

$R_{ijji}:= R(F_i, F_j, F_j, F_i)
= -\zeta_i \zeta_j + F_i \Gamma_{jj}^i  - F_j \Gamma_{ij}^i -(\Gamma_{ii}^j)^2  -(\Gamma_{jj}^i)^2
+ \sum_{k \neq 1, i, j } \{ \Gamma_{jj}^k \Gamma_{ik}^i
   -\Gamma_{ij}^k\Gamma_{jk}^i  -( \Gamma_{ij}^k-\Gamma_{ji}^k) \Gamma_{kj}^i \}
$. From (\ref{x23}) we can get (\ref{y01}).



\end{proof}
If $ [i], [j], [k]$ are pairwise different, then from  Lemma \ref{abc60x} (ii) we can get
\begin{eqnarray} \label{s2}
 \Gamma_{ij}^k\Gamma_{jk}^i  +( \Gamma_{ij}^k-\Gamma_{ji}^k) \Gamma_{kj}^i =-2\Gamma_{ij}^k\Gamma_{ji}^k
 \end{eqnarray}
\subsection{Analysis of the vacuum static space equation for $R_{ii}$ with $i >m$}
In this subsection we study   $R_{ii}$  when $ i > m$.


Suppose $i >m$. By the definition of $m$ in Section 3, $\zeta_i =0$ and $F_i  =  e_i$  from (\ref{e07}).   By (\ref{x5}),  $R_{ii} = \frac{R}{n-1}$.
From Lemma \ref{77bx01},  we then get

\begin{equation} \label{eq09}
 \frac{R}{n-1} = \sum_{j>1, j \neq i} \{  F_i \Gamma_{jj}^i  +F_j \Gamma_{ii}^j  -(\Gamma_{ii}^j)^2  -(\Gamma_{jj}^i)^2- \sum_{k \neq 1, i, j }\Gamma_{jj}^k \Gamma_{ii}^k
 +\Gamma_{ij}^k\Gamma_{jk}^i +( \Gamma_{ij}^k-\Gamma_{ji}^k) \Gamma_{kj}^i \}.
\end{equation}
We analyze the terms in the RHS of  (\ref{eq09}).

If $j >m$, then $[j]=[i]$ and
from Lemma \ref{L10}, $\Gamma_{jj}^i =  {\gamma}_{jj}^i(x).$
Again from Lemma \ref{L10}, $\Gamma_{ij}^k =  {\gamma}_{ij}^k(x)$ when $i, j,k > m$.

If  $[j] \neq [i]$  and  $[k] =[j]$, then $ -\Gamma_{ij}^k\Gamma_{jk}^i  -( \Gamma_{ij}^k-\Gamma_{ji}^k) \Gamma_{kj}^i =0$ from (\ref{eq37}).

If  $[j] \neq [i]$  and   $[k] =[i]$, then  $ -\Gamma_{ij}^k\Gamma_{jk}^i  -( \Gamma_{ij}^k-\Gamma_{ji}^k) \Gamma_{kj}^i =0$ also from (\ref{eq37}).

If $ [i], [j], [k]$ are pairwise different, then  $ -\Gamma_{ij}^k\Gamma_{jk}^i  -( \Gamma_{ij}^k-\Gamma_{ji}^k) \Gamma_{kj}^i =2\Gamma_{ij}^k\Gamma_{ji}^k =\frac{ 2 {\gamma}_{ij}^k(x) {\gamma}_{ji}^k(x)}{ h+ c_j  } $ from (\ref{s2}) and Lemma \ref{L11}.

From the above discussion,
 we use  (\ref{eq37}) to rewrite (\ref{eq09})  as below.
\begin{eqnarray*} \frac{R}{n-1} =  \sum_{j \in [i], j \neq i} \{   e_i {\gamma}_{jj}^i + e_j {\gamma}_{ii}^j - ({\gamma}_{ii}^j)^2- ({\gamma}_{jj}^i)^2\}
  - \sum_{j \in [i], j \neq i} \{\sum_{k \neq 1, i, j , k \in [i]}  {\gamma}_{jj}^k {\gamma}_{ii}^k\}\\
 - \sum_{j \in [i], j \neq i} \{\sum_{k \neq 1, i, j, k \in [i]}
   {\gamma}_{ij}^k  {\gamma}_{jk}^i  + ( {\gamma}_{ij}^k- {\gamma}_{ji}^k) {\gamma}_{kj}^i \}  \hspace{3.3cm}\\
   +\sum_{1<j \notin [i]} \frac{1}{h+ c_j}\{\sum_{1<k \notin [i]\cup [j] }
 2{\gamma}_{ij}^k(x) {\gamma}_{ji}^k(x) \}. \hspace{3.5cm}
\end{eqnarray*}
In the above, all terms except the last are independent of $s$. By Lemma \ref{L41} below, we obtain

\begin{lemma} \label{jf1}   For $i >m$ and for each fixed $[l] \neq [i]$, it holds that

 $\sum_{j \in [l] }  \sum_{1<k \notin [i]\cup [j]   }{\gamma}_{ji}^k(x){\gamma}_{ij}^k(x) =0. $  \end{lemma}
The next lemma is elementary; to prove it, one can multiply  by some polynomials with factors $(x+ c_i)$ and compare LHS and RHS.
\begin{lemma} \label{L41}
Let  $n_0$ and  $k_0$ be natural numbers and $c_i$ be pairwise  different  numbers.
Suppose that $\sum_{i=1}^{n_0} \sum_{t=1}^{k_0} \frac{\alpha_{i, t}}{(x+ c_i)^t} = b_0 + b_1 x$, for any $x$ in a non-empty open interval $I$, and constants $b_0,  b_1$, $\alpha_{i, t} $. Then
$\alpha_{i, t} =b_0 = b_1 =0, $ for any $1 \leq i \leq n_0$ and any    $1 \leq t \leq k_0$.
\end{lemma}



\subsection{Analysis of the vacuum static space equation for $R_{ii}$ with $i \leq m$}
In this subsection we study   $R_{ii}$  when $2 \leq i \leq m$.
From (\ref{eq37}),  $R_{ii}$ in       Lemma \ref{77bx01} becomes
 \begin{eqnarray}
 R_{ii}  = -\zeta_i^{'}  -  \zeta_i^2  + \sum_{1<j\leq m, j \neq i} \{  -\zeta_i \zeta_j + F_i \Gamma_{jj}^i  + F_j \Gamma_{ii}^j -(\Gamma_{ii}^j)^2  -(\Gamma_{jj}^i)^2 \} \nonumber \\
-\sum_{1<j\leq m, j \neq i}\{ \sum_{k \neq 1, i, j } \Gamma_{jj}^k \Gamma_{ii}^k
  +\Gamma_{ij}^k\Gamma_{jk}^i +( \Gamma_{ij}^k-\Gamma_{ji}^k) \Gamma_{kj}^i \} \hspace{1.5cm} \label{rii} \\
    - \sum_{j> m} \{     \sum_{k \neq 1, i, j }
   \Gamma_{ij}^k\Gamma_{jk}^i  +( \Gamma_{ij}^k-\Gamma_{ji}^k) \Gamma_{kj}^i \}  \hspace{3.4cm} \nonumber
\end{eqnarray}
We shall treat  terms in the RHS of (\ref{rii}).
For $2 \leq i \leq m$,   from (\ref{e07}) and  (\ref{fe1}),
we have $F_i   =  \frac{a_i}{h + c_i}   e_i$.
 Using (\ref{eq37}) and
    $\Gamma_{ii}^j = \frac{ {\gamma}_{ii}^j }{h+ c_i}$ for $j \in [i]$ by Lemma \ref{L10}, we get
 \begin{eqnarray} \label{aq01}
\sum_{1<j\leq m, j \neq i} \{  F_i \Gamma_{jj}^i  + F_j \Gamma_{ii}^j -(\Gamma_{ii}^j)^2  -(\Gamma_{jj}^i)^2 \}\hspace{3cm}  \\
 =\frac{1 }{(h+ c_i)^2} \sum_{j \in [i], j \neq i} \{  a_i e_i {\gamma}_{jj}^i  + a_j e_j {\gamma}_{ii}^j -  ({\gamma}_{ii}^j)^2-  ({\gamma}_{jj}^i)^2\} . \hspace{0cm} \nonumber
\end{eqnarray}
Using (\ref{eq37}) and Lemma \ref{L10}, we also get
 \begin{eqnarray}\label{aq02}
\sum_{1<j\leq m, j \neq i} \{ \sum_{k \neq 1, i, j }  \Gamma_{jj}^k \Gamma_{ii}^k \}=\sum_{j \in [i], j \neq i} \{\sum_{k \neq 1, i, j , k \in [i]} \frac{{\gamma}_{jj}^k }{h+ c_j} \frac{{\gamma}_{ii}^k }{h+ c_i} \}.
\end{eqnarray}
From (\ref{eq37}), (\ref{s2}), Lemma \ref{L10} and  Lemma \ref{L11}, we have
 \begin{eqnarray}\label{aq03}
 \\
\sum_{1<j \leq m, j \neq i}  \{ \sum_{k \neq 1, i, j } \Gamma_{ij}^k\Gamma_{jk}^i  +( \Gamma_{ij}^k-\Gamma_{ji}^k) \Gamma_{kj}^i \}   \hspace{6.5cm}  \nonumber\\
  = \sum_{j \in [i], j \neq i} \{\sum_{k \neq 1, i, j, k \in [i]}
   \Gamma_{ij}^k\Gamma_{jk}^i  +( \Gamma_{ij}^k-\Gamma_{ji}^k) \Gamma_{kj}^i \}
 -\sum_{1<j \notin [i]}^m \{\sum_{1< k \notin [i]\cup [j]}
   2\Gamma_{ji}^k \Gamma_{ij}^k \} \nonumber  \hspace{1.4cm}  \\
     = \sum_{j \in [i], j \neq i} \{\sum_{k \neq 1, i, j, k \in [i]}
   \frac{{\gamma}_{ij}^k  {\gamma}_{jk}^i + {\gamma}_{ij}^k{\gamma}_{kj}^i - {\gamma}_{ji}^k{\gamma}_{kj}^i }{(h+ c_i)^2}  \}
   -\sum_{1<j \notin [i]}^m \{\sum_{1< k \notin [i]\cup [j] }
   2\frac{ {\gamma}_{ji}^k(x)}{ h+ c_j  } \frac{ {\gamma}_{ij}^k(x)}{ h+ c_i  } \}. \nonumber
\end{eqnarray}
Using (\ref{eq37}), (\ref{s2}) and  Lemma \ref{L11}, we get
 \begin{eqnarray}\label{aq04}
 \sum_{j> m} \{     \sum_{k \neq 1, i, j }
   \Gamma_{ij}^k\Gamma_{jk}^i + ( \Gamma_{ij}^k-\Gamma_{ji}^k) \Gamma_{kj}^i \} \hspace{3.5cm}  \\
   =-\sum_{j> m} \{     \sum_{1< k \notin [i]\cup [j]}
   2\Gamma_{ij}^k\Gamma_{ji}^k \}
    =- \sum_{j >m} \{\sum_{1< k \notin [i]\cup [j] }
     2\frac{{\gamma}_{ij}^k }{h+ c_i}   {\gamma}_{ji}^k  \}.  \nonumber
\end{eqnarray}
For $2 \leq i \leq m$, from (\ref{aq01})$\sim$(\ref{aq04}),  (\ref{x4}) and  (\ref{x5}),  the formula (\ref{rii}) becomes
 \begin{eqnarray} \label{aq7}
 \\
\frac{h^{''} }{ h  +c_i } +  \frac{R}{n-1}
   = -\frac{h^{''} }{ h  +c_i }  + \sum_{1<j\leq m, j \neq i}  \{  -\frac{h^{'} }{h+ c_i}  \frac{h^{'} }{h+ c_j} \}  \hspace{5cm}  \nonumber \\
   +\frac{1 }{(h+ c_i)^2} \sum_{j \in [i], j \neq i} \{  a_ie_i {\gamma}_{jj}^i  +a_j e_j {\gamma}_{ii}^j -  ({\gamma}_{ii}^j)^2-  ({\gamma}_{jj}^i)^2\} \hspace{4.5cm}  \nonumber \\
   - \sum_{j \in [i], j \neq i} \{\sum_{k \neq 1, i, j , k \in [i]} \frac{{\gamma}_{jj}^k }{h+ c_j} \frac{{\gamma}_{ii}^k }{h+ c_i} \}
- \sum_{j \in [i], j \neq i} \{\sum_{k \neq 1, i, j, k \in [i]}
   \frac{{\gamma}_{ij}^k  {\gamma}_{jk}^i +{\gamma}_{ij}^k{\gamma}_{kj}^i -{\gamma}_{ji}^k{\gamma}_{kj}^i }{(h+ c_i)^2}  \} \nonumber   \hspace{1.5cm}  \\
   +\sum_{1<j \notin [i]}^m \{\sum_{1< k \notin [i] \cup [j] }
   2\frac{ {\gamma}_{ji}^k(x)}{ h+ c_j  } \frac{ {\gamma}_{ij}^k(x)}{ h+ c_i  } \}
   +\sum_{j >m} \{\sum_{1< k \notin [i] \cup [j] }
     2\frac{{\gamma}_{ij}^k }{h+ c_i}   {\gamma}_{ji}^k  \}   \hspace{3.8cm}  \nonumber
\end{eqnarray}


 \bigskip
Multiplying (\ref{aq7})  by $h+ c_i$,  we use  (\ref{ff01}) and  (\ref{ff02}) to have

\begin{eqnarray}  \label{30b92}
\\
 -\frac{2RQ}{(n-1)H} +\frac{2a}{H} +  \frac{R(h + c_i )}{n-1}  \hspace{8.5cm} \nonumber \\
  + \sum_{1<j\leq m, j \neq i}   \frac{1 }{(h+ c_j) } \{(h^{'}(0))^2- \frac{2R}{n-1}\int_{h(0)}^{h(s)}  \frac{Q}{H}dh +\int_{h(0)}^{h(s)} \frac{2a}{H} dh  \} \hspace{1.5cm} \nonumber \\
=\frac{1 }{(h+ c_i)} \sum_{j \in [i], j \neq i} \{ a_i e_i {\gamma}_{jj}^i  +a_j e_j {\gamma}_{ii}^j - ({\gamma}_{ii}^j)^2-  ({\gamma}_{jj}^i)^2\}  - \sum_{j \in [i], j \neq i} \{ \sum_{k \neq 1, i, j , k \in [i]}\frac{{\gamma}_{jj}^k {\gamma}_{ii}^k  }{h+ c_j} \} \nonumber\\
  -  \sum_{j \in [i], j \neq i} \frac{1}{h+c_i} \{\sum_{k \neq 1, i, j, k \in [i]}
     {\gamma}_{ij}^k{\gamma}_{jk}^i  +  {\gamma}_{ij}^k{\gamma}_{kj}^i - {\gamma}_{ji}^k {\gamma}_{kj}^i \}   \nonumber  \hspace{4cm} \\
   +\sum_{1<j \notin [i]}^m  \{ \sum_{1< k \notin [i] \cup [j] }
   2 \frac{ {\gamma}_{ji}^k(x){\gamma}_{ij}^k(x)}{ h+ c_j }\}
  +\sum_{j >m} \{  \sum_{1< k \notin [i] \cup [j] }
     2  {\gamma}_{ij}^k  {\gamma}_{ji}^k \}.\nonumber \hspace{2.8cm}
   \end{eqnarray}

We shall write in detail the LHS of (\ref{30b92}).
Let
$H= \Pi_{i=2}^m (h+ c_i) $ expand to $H= h^{m-1} + \alpha_{m-2} h^{m-2}   + \cdots + \alpha_0 $ with $\alpha_{m-2} =  \sum_{i=2}^m c_i$.

Let $[m_1], \cdots,  [m_d]$ be all the distinct equivalence classes with $2 \leq m_1  <\cdots < m_d \leq m $. For simplicity we denote $[m_l]$ by $[l]$ so that for $l =1, \cdots,  d$, we write  $c_{[l]}:= c_j$ for any $j$ in $[m_l]$.  Let $n_{[l]}$ be the dimension of $\lambda_{m_l}$-eigenspace so that $\sum_{l=1}^d  n_{[l]}= m-1$. Then we can write
$H = \Pi_{l=1}^d (h+ c_{[l]})^{n_{[l]}}  $. For its reciprocal $\frac{1}{H}$, we can do the rational function expansion;
\begin{eqnarray}  \label{30g6}
\frac{1}{H} =  \sum_{l=1}^d  \sum_{t=1}^{n_{[l]}} \frac{\alpha_{l, t}}{(h+ c_{[l]})^{t}},\end{eqnarray}
where  $a_{l, t}$ is a number depending on $l$ and $t$. Note that $\alpha_{l, n_{[l]}} \neq 0$.  Integration gives, for a constant $\tilde{c}_1$,

$\int_{h(0)}^{h(s)} \frac{1}{H} dh  = \sum_{l=1}^d \{  \sum_{t=2}^{n_{[l]}} \frac{\alpha_{l, t}}{(1-t )(h+ c_l)^{t-1}}+   \alpha_{l,1} \ln |h+ c_{[l]}| \} + \tilde{c}_1.$

 As $Q(h)$ is a function of $h$ such that
$\frac{dQ}{dh} = H$, we write
$Q=  \frac{h^m}{m} +\alpha_{m-2}\frac{h^{m-1}}{m-1} + \cdots + \alpha_0 h$. Division gives
$\frac{Q}{H}  = \frac{\frac{h^m}{m} +\alpha_{m-2}\frac{h^{m-1}}{m-1} + \cdots + \alpha_0 h}{h^{m-1} + \alpha_{m-2} h^{m-2}   + \cdots + \alpha_0}$, so we can write the  following rational function expansion
\begin{eqnarray}  \label{30b97}
\frac{Q}{H}  =   \frac{h}{m}  + \frac{\alpha_{m-2}}{m(m-1)}+  \sum_{l=1 }^d   \sum_{t=1}^{n_{[l]}} \frac{b_{l, t}}{(h+ c_{[l]})^{t}},
   \end{eqnarray}
 where  $b_{l, t}$ is a number depending on $l$ and $t$. Integration gives, for a constant $\tilde{c}_2$,

 $\int_{h(0)}^{h(s)} \frac{Q}{H} dh  =  \frac{h^2}{2m}  + \frac{\alpha_{m-2} h}{m(m-1)}+   \sum_{l=1 }^d \{  \sum_{t=2}^{n_{[l]}} \frac{b_{l, t}}{(1-t )(h+ c_{[l]})^{t-1}}+   b_{l,1} \ln |h+ c_{[l]}| \}+ \tilde{c}_2$.

\bigskip
From above discussion, the LHS of (\ref{30b92})  equals
\begin{eqnarray}  \label{803}
 \hspace{2.5cm} \\
-2 \frac{R}{(n-1)}\{ \frac{h}{m}  + \frac{\alpha_{m-2}}{m(m-1)}+  \sum_{l=1 }^d   \sum_{t=1}^{n_{[l]}} \frac{b_{l, t}}{(h+ c_{[l]})^{t}} \}
 +2a \{   \sum_{l=1}^d  \sum_{t=1}^{n_{[l]}} \frac{\alpha_{l, t}}{(h+ c_{[l]})^{t}} \}\nonumber \hspace{0.8cm} \\ +  \frac{R(h + c_i )}{n-1}
+\sum_{1<j\leq m, j \neq i}   \frac{ (h^{'}(0))^2 }{(h+ c_j) } -\sum_{1<j\leq m, j \neq i}   \frac{2R }{(h+ c_j) (n-1)} \big[  \frac{h^2}{2m}  + \frac{\alpha_{m-2} h}{m(m-1)} \big] \nonumber  \hspace{0cm} \\
-\sum_{1<j\leq m, j \neq i}   \frac{2R }{(h+ c_j) (n-1)} \big[     \sum_{l=1 }^d \{  \sum_{t=2}^{n_{[l]}} \frac{b_{l, t}}{(1-t )(h+ c_{[l]})^{t-1}}+   b_{l,1} \ln |h+ c_{[l]}| \}+ \tilde{c}_2 \big] \nonumber \\
 +\sum_{1<j\leq m, j \neq i}   \frac{2a }{(h+ c_j) }\big[\sum_{l=1}^d \{  \sum_{t=2}^{n_{[l]}} \frac{\alpha_{l, t}}{(1-t )(h+ c_{[l]})^{t-1}}+   \alpha_{l,1} \ln |h+ c_{[l]}| \} + \tilde{c}_1   \big]   . \nonumber \hspace{0.8cm}
   \end{eqnarray}

 (\ref{803}) is of the form $ \sum_{l=1}^d \{ \sum_{t=1}^{n_{[l]}} \frac{\gamma_{l, t}  }{(h+ c_{[l]} )^t } \} + \sum_{l=1}^d  \sum_{s=1}^d  \frac{ \delta_{l, s} \ln |h+ c_{[l]}|}{(h+ c_{[s]} )}+  \gamma_0 + \gamma_1 h $, where $\gamma_0$, $ \gamma_1$, $\gamma_{l, t}$ and $ \delta_{l, s}$  are constants. Here we use the fact  that if $c_j \neq c_p$,  then $\frac{1}{(h+c_j)(h+ c_p)^{t}} = \frac{a_0}{(h+c_j)}  + \sum_{k=1}^t \frac{a_k}{(h+ c_p)^{k}} $ for constants $a_0, \cdots, a_t$.

\bigskip

Meanwhile,  the  RHS of (\ref{30b92}) is of the form $ \sum_{l=1}^d   \frac{\beta_{l}(x) }{(h+ c_{[l]})} + \beta_0(x) $, where the functions $\beta_l(x)$ and $ \beta_0(x)$ are independent of $h$.

\section{Classification of vacuum static spaces with harmonic curvature}

In this section we shall prove Theorem \ref{locals}.
Equating LHS=RHS  of (\ref{30b92}), we   have

\begin{equation}  \label{30b98}
\sum_{l=1}^d \{ \sum_{t=1}^{n_{[l]}} \frac{\gamma_{l, t}  }{(h+ c_{[l]} )^t } \} + \sum_{l=1}^d \sum_{s=1}^d  \frac{ \delta_{l, s} \ln |h+ c_{[l]}|}{(h+ c_{[s]}) }+  \gamma_0 + \gamma_1 h= \sum_{l=1}^d   \frac{\beta_{l}(x)}{(h+ c_{[l]})} + \beta_0(x).
   \end{equation}
LHS is independent of $x$. Take derivative $\frac{\partial }{\partial x_i}$ on (\ref{30b98}) and apply Lemma  \ref{L41} to see  that $\beta_l(x)$ and $ \beta_0(x)$ are constants. Next, we need


\begin{lemma} \label{rel}
Let $R(x)$ and $p_l (x)$, $l=1, \cdots , k$, be rational functions defined on a non-empty open interval $I \subset \mathbb{R}, $
and
let $c_l$ be pairwise different real numbers.
If $R(x) = \sum_{l=1}^k  p_l(x) \ln|x+ c_l|$ for all $x$ on $I$, then $R(x)=0$ and $p_l(x) =0$.

\end{lemma}

\begin{proof} We  first suppose that all of  $p_l(x)$ are   non-zero polynomials. In this case,
if $m \geq 0$ is the maximum among the degrees of $ p_1, \cdots,   p_l$, we  differentiate the equality $m$ times to get
$ \sum_{l=1}^k  p_l^{(m)}(x) \ln|x+ c_l| = R_m(x)$  for a rational function $R_m(x)$. Let $ p_l^{(m)}(x)= a_l$. Some $a_l$ are nonzero and we have
$ \sum_{l=1}^k  a_l \ln|x+ c_l| =R_m(x)$ on $I$. This quickly leads to $a_l=0$ and  $p_l(x) =0$ for all $l$ and
then $R(x)=0$.


\smallskip
Now suppose that $p_l$ are rational functions.
 Let $P(x)$ be  the least common multiple of the denominators of $p_l$.
We get $P(x)R(x) = \sum_{l=1}^k  P(x)p_l(x) \ln|x+ c_l|$. Now $P \cdot p_l$ is a polynomial.
By the above paragraph,  we know that  $P(x)R(x)= P(x)p_l(x) =0$.  As $P  \neq 0$, we get  $R(x)=0$ and $p_l(x) =0$.

\end{proof}

Apply Lemma \ref{rel} and Lemma \ref{L41} to (\ref{30b98}), and get $\delta_{l, s} =0 $,
$\gamma_1 =0$,   $\gamma_0 =\beta_0$,  $\gamma_{l, t}=0$ for $t \geq 2$, $\gamma_{l, 1}=\beta_{l}$.
From (\ref{30b92}) and   (\ref{803}),  the equality $\gamma_0 =\beta_0$ gives


\begin{eqnarray}  \label{80a}
\frac{-2R}{(n-1)}  \frac{\alpha_{m-2}}{m(m-1)}  +  \frac{R c_i }{n-1}  - \frac{2R}{n-1}  \sum_{1<j\leq m, j \neq i}(  \frac{-c_j}{2m}  + \frac{\alpha_{m-2} }{m(m-1)}) \nonumber \\
= \sum_{j >m} \{ \sum_{1< k \notin [i] \cup [j] }
     2  {\gamma}_{ij}^k  {\gamma}_{ji}^k \} \hspace{2cm}
   \end{eqnarray}
If $m=n$,   then the RHS of (\ref{80a}) vanishes.  If $m < n$,
we use Lemma \ref{jf1};    for each $[l]$ with $2 \leq m_l\leq m$,
$\sum_{i \in [l]} \sum_{j >m} \sum_{1< k \notin [i] \cup [j] }
      {\gamma}_{ij}^k  {\gamma}_{ji}^k=0.$  So, regardless of $m=n$ or  $m < n$, we get from (\ref{80a})

$\frac{-R}{(n-1)}\sum_{i \in [l]}   \big[   \frac{2\alpha_{m-2}}{m(m-1)} -  c_i   - \sum_{1<j\leq m, j \neq i} \frac{ c_j }{m}  + (m-2)\frac{2\alpha_{m-2} }{m(m-1)} \big] = 0.$


\medskip
\noindent
  Suppose $R \neq 0$. We can get
$  2 \alpha_{m-2}n_{[l]}-  c_{[l]} n_{[l]} m  - \sum_{i \in [l]}\{  -c_i   + \alpha_{m-2}   \}
=0,$ as $\alpha_{m-2}=\sum_{j=2}^m c_j .$
Then $c_{[l]} (m-1)= \alpha_{m-2}=\sum_{j=2}^m c_j $.  For two equivalent classes $[l_1] $ and $[l_2]$ with $2 \leq m_{l_k} \leq m$, $k=1,2$,  we get $c_{[l_1]} (m-1) = \alpha_{m-2}= c_{[l_2]} (m-1)$. So,  $c_{[l_1]}  = c_{[l_2]} $ and $[l_1] = [l_2]$.
Therefore there  exists exactly one distinct nonzero $\zeta_i$.
We proved
   \begin{lemma} \label{jf12} Suppose that there exists some nonzero $\zeta_i$.
  If $R \neq 0$, then there exists exactly one distinct nonzero $\zeta_i$.
   \end{lemma}


Now  we assume $R=0$.
Then (\ref{30b92}) and (\ref{803})   give
 \begin{eqnarray}  \label{80379h}
  \\
 2a    \sum_{l=1}^d  \sum_{t=1}^{n_{[l]}} \frac{\alpha_{l, t}}{(h+ c_{[l]})^{t}}   +\sum_{1<j\leq m, j \neq i}   \frac{1 }{(h+ c_j) } (h^{'}(0))^2  \hspace{5cm} \nonumber\\
 +\sum_{1<j\leq m, j \neq i}   \frac{1 }{(h+ c_j) } \big[ 2a\sum_{l=1}^d \{  \sum_{t=2}^{n_{[l]}} \frac{\alpha_{l, t}}{(1-t )(h+ c_{[l]})^{t-1}}+   \alpha_{l,1} \ln |h+ c_{[l]}| \} + 2a \tilde{c}_1      \big] \nonumber \\
 = \sum_{l=1}^d   \frac{\beta_{l}}{(h+ c_{[l]})} + \beta_0. \hspace{9.5cm} \nonumber   \end{eqnarray}

 From Lemma \ref{rel}, we get $ a \alpha_{l, 1}=0$ for each $l$,  and $\beta_0 =0$
If $a=0$, (\ref{ff01}) gives $h^{''} = f^{'} =0$, a contradiction to that $f$  is not a constant.

\smallskip
Assume $a \neq 0$. Then  $ \alpha_{l,1}=0$ for each $l$. By $H = \Pi_{l=1}^d (h+ c_{[l]})^{n_{[l]}}  $ and  (\ref{30g6}),   $n_{[l]} \geq 2$ for each $ l =1,  \cdots d$.
We write $ \sum_{1<j\leq m, j \neq i} \frac{1  }{(h+ c_j) }
 = \sum_{r=1}^d  \frac{ n_{[r]}}{(h+ c_{[r]}) } -    \frac{1}{h+c_{i}}$.

Now (\ref{80379h})  becomes
   \begin{eqnarray}  \label{80379d}
   \\
   \sum_{l=1}^d  \sum_{t=2}^{n_{[l]}} \frac{2a\alpha_{l, t}}{(h+ c_{[l]})^{t}} +  \sum_{1<j\leq m, j \neq i}   \frac{(h^{'}(0))^2 +2a \tilde{c}_1 }{(h+ c_j) }   \hspace{5cm} \nonumber\\
    +
     \sum_{r=1}^d \sum_{l=1}^d\sum_{t=2}^{n_{[l]}} \frac{ n_{[r]}}{(h+ c_{[r]}) } \frac{2a \alpha_{l, t}}{(1-t )(h+ c_{[l]})^{t-1}}-    \sum_{l=1}^d \sum_{t=2}^{n_{[l]}} \frac{1}{h+c_{i}}\frac{2a \alpha_{l, t}}{(1-t )(h+ c_{[l]})^{t-1}} \nonumber  \\
 = \sum_{l=1}^d   \frac{\beta_{l}}{(h+ c_{[l]})}. \hspace{10cm} \nonumber
   \end{eqnarray}


\medskip
Suppose that $ d \geq 2$.
As $n_{[l]} \geq 2$, for fixed $[l] \neq [i]$, the coefficient  of $\frac{1}{(h+ c_{[l]})^{n_{[l]}}}$
 in the LHS of  (\ref{80379d}) equals
$(2a  \alpha_{l, n_{[l]}}) (1+  \frac{n_{[l]} }{1- n_{[l]}} )$, which should vanish by Lemma \ref{L41}.
We obtain $ \alpha_{l, n_{[l]}} =0.$ But as defined in (\ref{30g6}),  $ \alpha_{l, n_{[l]}}\neq 0$, a contradiction.  We proved that $d=1$.
 We write
\begin{lemma} \label{ph0}
Suppose that there exists some nonzero $\zeta_i$.
  If $R = 0$, then there exists exactly one distinct nonzero $\zeta_i$.
\end{lemma}

Lemma \ref{jf12} and Lemma \ref{ph0} give

\begin{prop} \label{ph1}
Suppose that there exists some nonzero $\zeta_i$.
Then there exists exactly one distinct nonzero $\zeta_i$.
\end{prop}

{\bf Proof of Theorem \ref{locals}}:
Recall from (\ref{F32b2}) that
if  $\zeta_i =0$ for all $i=2, \cdots, n$, then $(M, g)$ has at most two distinct Ricci eigenvalues at each point, say $\lambda_1$ and $\lambda_2$.
So, by Proposition \ref{ph1}, a vacuum static space $(M, g, f)$ with harmonic curvature can have at most three distinct Ricci eigenvalues, i.e. $\lambda_1,  \lambda_2, \lambda_n$; here $ \lambda_n$ corresponds to $\zeta_n=0$.
We can use the classification  result  of vacuum static spaces with harmonic curvature and  at most three distinct Ricci eigenvalues  in \cite{Ki2, Ki3}.
Therefore, we proved Theorem \ref{locals}.

\medskip
{\bf Proof of Corollary \ref{complete}}: This corollary can be proved from Theorem \ref{locals} and the real analyticity of $(M, g, f)$ by a standard argument. We omit the details.



\bigskip \noindent
\footnotesize{ Jongsu Kim: Dept. of Math., Sogang University, Seoul, Korea; \ \ jskim@sogang.ac.kr}

\end{document}